\documentclass[12pt]{amsart}

\usepackage{amsmath}

\usepackage{amsfonts,amssymb,amsthm,enumitem}
\usepackage[margin=1.2in]{geometry}
\usepackage{mathtools}
\usepackage{nicefrac,setspace}
\usepackage{xcolor}
\usepackage{hyperref}       
\usepackage{url}            
\usepackage{booktabs}       
\usepackage{amsfonts}       
\usepackage{nicefrac}       
\usepackage{microtype}      

\newtheorem{theorem}{Theorem}
\newtheorem{claim}[theorem]{Claim}
\newtheorem*{theorem*}{Theorem}
\newtheorem*{claim*}{Claim}
\newtheorem*{remark*}{Remark}

\newtheorem*{lemma*}{Lemma}

\newtheorem{lemma}[theorem]{Lemma}

\newcommand{\egh}{{\mathsf{EGH}}}

\renewcommand{\mod}{\mathop {\mathsf{mod}}}

\newcommand{\Z}{{\mathbb Z}}

\newcommand{\calX}{{\mathcal{X}}}

\newcommand{\good}{\mathbb{{T}}}

\newcommand{\N}{\mathbb{N}}
\newcommand{\eps}{\epsilon}

\newcommand{\R}{\mathbb{R}}

\newcommand{\F}{\mathbb{F}}
\newcommand{\E}{\mathop \mathbb{E}}

\newcommand{\gh}{{\mathsf{GH}}}
\newcommand{\nbits}{\{\pm 1\}}
\newcommand{\inner}[2]{\left \langle #1 , #2 \right \rangle}

\newcommand{\ip}[2]{\left\langle #1,#2\right\rangle}

\newcommand{\Expect}[1]{\mathop{\mathbb{E}}\left
	[ #1 \right ]}
\newcommand{\Ex}[2]{\mathop{\mathbb{E}}\displaylimits_{#1}\left
	[ #2 \right ]}

\title{Anti-concentration and \\ the Exact Gap-Hamming problem}

\begin{document}

\author{Anup Rao}
\address{School of Computer Science, University of Washington}
\email{anuprao@cs.washington.edu}
\author{Amir Yehudayoff}
\address{Department of Mathematics, Technion-IIT}
\email{amir.yehudayoff@gmail.com}
\thanks{A.Y.\ is partially supported by ISF grant 1162/15.
This work was done while
the authors were visiting the Simons
Institute for the Theory of Computing.}

\begin{abstract}
We prove anti-concentration bounds 
for the inner product of two independent random vectors, and use these bounds to prove lower bounds in communication complexity.  
We show that if $A,B$ are subsets of the cube
$\{\pm 1\}^n$ with $|A| \cdot |B| \geq 2^{1.01 n}$, and $X \in A$ and $Y \in B$ are sampled independently and uniformly, 
then the inner product $\ip{X}{Y}$ takes on any fixed value with probability at most $O(1/\sqrt{n})$. 
In fact, we prove the following stronger ``smoothness'' statement:
$$ \max_{k } \big| \Pr[\ip{X}{Y} = k] - \Pr[\ip{X}{Y} = k+4]\big|
\leq O(1/n).$$ 
We use these results to prove that the exact gap-hamming problem requires linear communication, resolving an open problem in communication complexity. 
We also conclude anti-concentration for
structured distributions with low entropy.
If $x \in \Z^n$ has no zero coordinates, and $B \subseteq \{\pm 1\}^n$ corresponds to a subspace of $\F_2^n$ of dimension $0.51n$, then 
$\max_k \Pr[\ip{x}{Y} = k]  \leq O(\sqrt{\ln (n)/n})$.
\end{abstract}

\maketitle

\section{Introduction}
Anti-concentration bounds establish that the distribution of outcomes of a random process is 
not concentrated in any small region. No single outcome is obtained too often.
Anti-concentration plays an important role in 
mathematics and computer
science. It is used in the  study of roots of random  polynomials~\cite{littlewood1943number},
 random matrix theory~\cite{kahn,tao2009inverse},
 communication complexity~\cite{chakrabarti2012optimal,vidick2012concentration,sherstov2012communication},
 quantum computation~\cite{Aaronson_2011},
and more. In particular, as we discuss below, anti-concentration bounds are very useful to understand the communication complexity of the gap-hamming function.

A well-known context in which anti-concentration has been studied extensively  is the sum of independent identically distributed random
variables. If $Y \in \nbits^n$ is uniformly distributed, then 
the probability that $\sum_{j=1}^n Y_j$ takes any specific value
is at most $\binom{n}{\lceil n/2 \rceil}/2^n = O(\tfrac{1}{\sqrt{n}})$. This was studied and generalized 
by Littlewood and Offord~\cite{littlewood1943number},
Erd\H{o}s~\cite{erdos1945lemma}, 
and many others. The classical Littlewood-Offord problem is about understanding the anti-concentration of the inner product $\ip{x}{Y} = \sum_{j=1}^n x_j Y_j$, for arbitrary $x \in \R^n$ and $Y \in \nbits^n$ chosen uniformly. For example, Erd\H{o}s proved that if $x$ has no non-zero coordinates, then  $\max_k  \Pr[\ip{x}{Y} = k] \leq \binom{n}{\lceil n/2 \rceil}/2^n = O(\tfrac{1}{\sqrt{n}})$.

It is interesting to understand the most general conditions under which such anti-concentration holds. Various generalizations were 
studied by Frankl and F\"uredi~\cite{frankl1988solution}, 
Hal{\'a}sz~\cite{halasz1977estimates}
and others (see~\cite{tao2010sharp}
and references within). These results show that stronger bounds can be proved when the vector $x$ satisfies stronger conditions. In this past work, the vector $Y$ is typically assumed to be uniformly distributed; indeed anti-concentration  fails when the 
entries of $Y$ are not independent.
For example, if $Y$ is sampled uniformly from the set of strings with exactly $\lceil n/2 \rceil$ entries that are $1$, then for $x = 1^n$,  
we have that $\ip{x}{Y}$ is always the same. Can we somehow recover anti-concentration when $Y$ is not uniform? 

We show that if extra structure holds then anti-concentration 
is still recovered although the entropy is small.
For example, if we identify $\{\pm 1\}^n$ with the vector space $\F_2^n$, by associating $-1$ with $1$ and $1$ with $0$, then our results imply:
\begin{theorem}\label{thm:subspace} 
There exists a constant $C>0$ so that the following holds.
	If $x \in \Z^n$ has no zero coordinates, $B \subseteq \{\pm 1\}^n$ corresponds to a subspace of $\F_2^n$ of dimension $0.51n$, and $Y \in B$ is uniformly distributed, then $$\max_{k \in \Z} \Pr[\ip{x}{Y} = k] \leq C \sqrt{\tfrac{\ln n}{n}}.$$
\end{theorem}
Theorem \ref{thm:subspace} is a direct consequence of Theorem \ref{thm:main1} below. It shows that anti-concentration holds even when $Y$ is far from being uniform,
but when the direction $x$ is random as well.

\begin{remark*} 
	Theorem \ref{thm:subspace} and similar results can be used as a black box to prove the same bounds when $x$ is a real-valued vector. To see this, think of the relevant real numbers as vectors in a finite dimensional vector space
	over the rationals. We omit the details here.
\end{remark*}

Another natural setting is to consider the inner-product $\ip{X}{Y}$ of two independent variables, neither of which may be uniform. Recent work has proved some interesting results under the assumption that $X,Y$ have nice structure \cite{tran2020smallest, jain2021singularity}, but what if the only assumption is that $X,Y$ a uniformly distributed on large sets?
The following theorem, proved by Chakrabarti and Regev ~\cite{chakrabarti2012optimal} along the way to proving new lower bounds in communication complexity, shows that this does recover some anti-concentration:

\begin{theorem*}[Chakrabarti and Regev~\cite{chakrabarti2012optimal}]
	There is a constant $c >0$ such that if $A,B \subseteq \nbits^n$ are each of size at least $2^{(1-c)n}$ and $X \in A, Y \in B$ are sampled uniformly and independently, then 
	$$\Pr [|\ip{X}{Y}| \leq c \sqrt{n}  ] \leq 1-c.$$
\end{theorem*}
Alternate proofs of the same bound were subsequently given in ~\cite{vidick2012concentration,sherstov2012communication}. The theorem shows that $\ip{X}{Y}$
does not land in an interval of length 
much smaller than $\sqrt{n}$ with high probability. The strongest anti-concentration bounds give point-wise estimates. We would like to control the
{\em concentration probability} 
$$\max_{k \in \Z}  \ \Pr[ \ip{X}{Y} = k] ;$$
see~\cite{tao2009inverse}
and references within.

In our work we prove
a sharp bound on the point-wise concentration probability that holds for an overwhelming majority of directions $x$.

\begin{theorem}\label{thm:main}
	For every $\beta > 0$  and 
	$\delta > 0$,
	there exists $C>0$ such that the following holds. If $B \subseteq \nbits^n$ is of size $2^{\beta n}$, and $Y \in B$ is uniformly distributed, then for all but $2^{n(1-\beta +\delta)}$ directions $x \in \nbits^n$, 
	$$\max_{k \in \Z} \Pr_Y[\ip{x}{Y}=k] \leq \tfrac{C}{\sqrt{n}}.$$ In particular, if $X$ is independent of $Y$ and uniformly distributed in a set $A$ of size $2^{n(1-\beta + 2\delta)}$, then
	$$\max_{k \in \Z} \Pr_Y[\ip{X}{Y}=k] \leq \tfrac{C}{\sqrt{n}} + 2^{-\delta n}.$$
\end{theorem}
Our bound implies the result of Chakrabarti and Regev, but it is strictly stronger. It is also tight in the following senses.
As mentioned above, the $O(\tfrac{1}{\sqrt{n}})$ bound is tight even when $A$ and $B$ are $\{\pm 1\}^n$.
To see that the bound on the number of bad directions is sharp\footnote{This example is due to an anonymous reviewer.}, observe that if $B \subset \nbits^n$ is the set of $y$'s with $\sum_{j=1}^n y_j = 0$, and $A \subset \nbits^n$ is the set of $x$'s with $\sum_{j=1}^n x_j = (1-2\epsilon) n$ for some small $\epsilon>0$, then 
$$|B| \approx \tfrac{1}{\sqrt{n}} 2^{n}
\quad \& \quad |A| \approx 2^{h(\epsilon) n},$$ where $h(\epsilon)$ is the binary entropy function. Yet for every $x \in A$,  
$$\Pr[|\ip{x}{Y}| \leq 1] \geq \Omega(\tfrac{1}{\sqrt{\epsilon n}}).$$ 
The sets $A,B$ do not satisfy the conclusions of Theorem \ref{thm:main}, even though $|A| \cdot |B| \approx 2^{(1+h(\epsilon) )n}$.

Our methods lead to even stronger conclusions about the distribution of $\inner{x}{Y}$. We prove the following \emph{smoothness} result:

\begin{theorem}
	\label{thm:mainP} For every $\beta, \epsilon>0$, 
	there is $C>0$ so that the following holds.
	Suppose $B \subseteq \{\pm 1\}^n$ is a set with $|B|= 2^{\beta n}$, and $Y \in B$ is uniformly distributed. Then for all but $2^{(1-\beta + \epsilon)n}$ choices of $x\in \{
	\pm 1\}^n$, we have:
	$$\max_{k \in \Z} \big| \Pr[\ip{x}{Y} = k] - \Pr[\ip{x}{Y} = k+4\big|
	\leq \frac{C }{n}.$$ In particular, if $X$ is independent of $Y$ and uniformly distributed in a set $A$ of size $2^{(1-\beta+2\epsilon)n}$, then $$\max_{k \in \Z} \big| \Pr[\ip{X}{Y} = k] - \Pr[\ip{X}{Y} = k+4\big|
	\leq \frac{C }{n} + 2^{-\epsilon n}.$$
\end{theorem}

Theorem \ref{thm:main} is implied by Theorem \ref{thm:mainP}. Indeed, if $x$ is such that $$\max_{k \in \Z} \big| \Pr[\ip{x}{Y} = k] - \Pr[\ip{x}{Y} = k+4\big|
\leq \frac{C }{n},$$ then for all $k$ and $j \leq m$, 
 $$\Pr[\ip{x}{Y} = k] \leq \Pr[\ip{x}{Y} = k+4j] +  \frac{Cm}{n}.$$ But then we must have (for $m \approx \sqrt{n}$):
$$ \Pr[\ip{x}{Y} = k]  \leq \frac{Cm}{n} + \frac{1}{m } \cdot \sum_{j=1}^{m} \Pr[\ip{x}{Y} = k+4j] \leq \frac{Cm}{n} + \frac{1}{m} \lesssim \frac{1}{\sqrt{n}}.$$ 

Theorem \ref{thm:mainP} is proved in Section~ \ref{sec:smooth}. 
It is sharp in the following two senses.
First, even for the case $A=B=\{\pm 1\}^n$,
there is a $k$ so that\footnote{For an integer $k = \tfrac{n}{2}-\sqrt{n}$, we have
	${n \choose k+1}-{n \choose k}
	= {n \choose k+1} \frac{n-2k-1}{n-k} 
	\gtrsim \frac{2^n}{n} $.}
$$|\Pr[\ip{X}{Y} = k] - \Pr[\ip{X}{Y} = k+4| \geq \Omega( \tfrac{1}{n}).$$
So, $O(\tfrac{1}{n})$ is the best upper bound
possible.
Secondly, if $A = \{x\in \{\pm 1\}^n: n- \sum_{j=1}^n x_j =0\mod 4\}$ and $B = \{y \in \{\pm 1\}^n: n- \sum_{j=1}^n y_j =0 \mod 4\}$, then because $n = \ip{x}{y} \mod 2$ for all $x,y \in \{\pm 1\}^n$,  $\ip{x}{y} \mod 4$ is the same for every pair $x \in A, y \in B$. Thus, there are sets $A,B$ with $|A| = |B| = 2^{n-1}$
so that for all $j \in \{1,2,3\}$,
$$\ \  \ \ |\Pr[\ip{X}{Y} = 0] - \Pr[\ip{X}{Y} = j]|
= \Pr[\ip{X}{Y} = 0]  = \Omega( \tfrac{1}{\sqrt{n}})$$
So, $4$ is the minimum gap for which an 
$O(\tfrac{1}{n})$ upper bound holds.

Our proof builds a flexible framework for proving anti-concentration results in discrete domains. We use this framework to show that anti-concentration holds in a wide variety of settings. As we explain below, we show that bounds similar to those proved in ~\cite{Erdos2,Sark,Stan,halasz1977estimates}
apply even when the underlying distribution is not uniform. When $Y$ is uniformly distributed, the additive structure of the entries in the direction vector  $x$  
controls anti-concentration~\cite{frankl1988solution}. If $x$ is unstructured, we get even stronger anti-concentration bounds for $\ip{x}{Y}$. 
This idea is instrumental when analyzing random
matrices~\cite{kahn,tao2009inverse}.

We choose the direction $x$
from sets of the following form. 
We call a set $A \subset \Z^n$ a two-cube
if $A = A_1 \times A_2 \times \cdots \times A_n$,
where each $A_j = \{u_j,v_j\}$ consists of two distinct integers. 
The differences of $A$ are the numbers $d_j = u_j - v_j$
for $j \in [n]$.

The following theorem describes three cases 
that yield different anti-concentration bounds.
It shows that the additive structure of 
$A$ is deeply related to the bounds we obtain.
The less structured $A$ is, the stronger the bounds are.
The first bound in the theorem holds for arbitrary two-cubes.
The second bound holds
when all the differences $d_1,\ldots,d_n$ are distinct.
The third bound applies in more general settings where 
the set of differences is unstructured.
This is captured by the following definition. 
A set $S \subset \N$ of size $n$ is called a Sidon set, or a Golomb ruler, if the number of solutions to the equation
$s_1+ s_2 = s_3+ s_4$ for $s_1,s_2,s_3,s_4 \in S$
is $4\cdot {n \choose 2}+n$.
In other words, every pair of integers has a distinct sum. 
Sidon sets were defined by Erd\H{o}s and Tur\'an~\cite{ET}
and have been studied by many others since.
We say that 
$S \subset \Z$ is a weak Sidon set
if the number of solutions to the equation
$\eps_1 s_1+ \eps_2 s_2 = \eps_3 s_3+\eps_4 s_4$ for 
$\eps_1,\ldots,\eps_4 \in \{\pm 1\}$ and
$s_1,\ldots,s_4 \in S$
is at most $100 n^2$. The number $100$ can be replaced by any other constant, we use it here just to be concrete.

\begin{theorem}\label{thm:main1}
For every $\beta > 0$ and $\delta >0$,
there exists $C>0$ such that the following holds. 
Let $A \subset \Z^n$ be a two-cube with differences
$d_1,\ldots,d_n$.
Let $B \subseteq \nbits^n$ be of size $2^{\beta n}$ and $Y$ be uniformly distributed in $B$.
\begin{enumerate}
\item For all but $2^{n(1-\beta +\delta)}$ directions $x \in A$, 
		$$\max_{k \in \Z} \Pr_Y\left[\ip{x}{Y} =k\right] \leq 
	C \sqrt{\tfrac{\ln(n)}{n}} .$$
\item If $d_1,\ldots,d_n$ are distinct, then
for all but $2^{n(1-\beta +\delta)}$ directions $x \in A$, 
		$$\max_{k \in \Z} \Pr_Y\left[\ip{x}{Y} =k\right] \leq 
	C \sqrt{\tfrac{\ln(n)}{n^3}}  .$$
\item If $\{d_1,\ldots, d_n\}$ is a weak Sidon set of size $n$, then
for all but $2^{n(1-\beta +\delta)}$ directions $x \in A$, 
		$$\max_{k \in \Z} \Pr_Y\left[\ip{x}{Y} =k\right] \leq 
	C \sqrt{\tfrac{\ln(n)}{n^5}} .$$
	\end{enumerate}
\end{theorem}


To see why this is a generalization of past work, observe that if $Y \in \{\pm 1\}^n$ is uniformly distributed, then for any $x \in \Z^n$, the distribution of $\inner{x}{Y}$ is identical to the distribution of $\inner{X}{Y}$, where $X$ is obtained by picking uniformly random signs for the coordinates of $x$. The number of directions in the support of $X$ is $2^n$, and the theorem above can be applied. 

A similar idea proves Theorem \ref{thm:subspace}. 
The key point is the assumption that $B$ corresponds to a subspace
of $\F_2^n$.
Every element of $B$ corresponds to a signing of $x$ that gives the same distribution for $\ip{x}{Y}$. 
We thus obtained a set $A$ of distinct directions of size $|A|=|B|$.
Because $|A| \cdot |B| \geq 2^{1.02n}$, 
we can apply Theorem \ref{thm:main1} to prove Theorem \ref{thm:subspace}.

The proof of  Theorem~\ref{thm:main1}
is given in Section~\ref{sec:UTC}. The first bound in Theorem~\ref{thm:main1} nearly implies Theorem~\ref{thm:main}. It is weaker by a factor of $\sqrt{\ln(n)}$. However, it holds for all two-cubes, not just the hypercube $\{\pm 1\}^n$. 
The second bound almost
matches the sharp $O(1/n^{1.5})$ bound 
that holds when $(u_j,v_j) = (j,-j)$ for each $j$ and $Y$ is uniform in the hypercube~\cite{Sark,Stan}.
We believe that the $\sqrt{\ln(n)}$ factor 
is not needed, but were not able to eliminate it.
The theorem is, in fact, part of a more general phenomenon.
We postpone the full technical description to
Section~\ref{sec:UTC}.

\subsection*{An application to Communication Complexity}
These kinds of anti-concentration bounds are intimately connected to understanding the communication complexity of the \emph{gap-hamming} function. The gap-hamming function
$\gh = \gh_{n,k} : \{\pm 1\}^n \rightarrow \{0,1,\star\}$ is defined
by
$$\gh(x,y) = \begin{cases}
1 & \ip{x}{y}\geq k, \\
0 & \ip{x}{y} \leq - k ,\\
\star & \text{otherwise.}
\end{cases}$$ Note that the Hamming distance between $x$ and $y$
is $\tfrac{n - \ip{x}{y}}{2}$.
This problem is well-studied in communication complexity; for background and definitions,
see the books~\cite{kushilevitzcommunication,raoY}. 
Alice gets $x$, Bob gets $y$, and their goal is to compute $\gh(x,y)$.
It is a promise problem; the protocol is allowed to compute any value when the input corresponds to a $\star$, and it needs to be correct only on the remaining inputs. 
The standard choice for $k$ is $\lceil  \sqrt{n} \rceil$,
so we write $\gh_n$ to denote $\gh_{n , \lceil \sqrt{n} \rceil}$.

The gap-hamming problem was introduced by Indyk and Woodruff in the context of streaming algorithms~\cite{indyk2003tight},
and was subsequently studied and used in many works and in various contexts (see~\cite{jayram2008one,woodruff2009average,brody2009multi,brody2010better,chakrabarti2010near} and references within).
Proving a sharp $\Omega(n)$ lower bound on its randomized communication
complexity was a central open problem for almost ten years,
until Chakrabarti and Regev~\cite{chakrabarti2012optimal} solved it using the anti-concentration bound mentioned above.
Later, Vidick~\cite{vidick2012concentration} and 
Sherstov~\cite{sherstov2012communication}  found simpler proofs.
The difficulties in proving this lower bound
are explained in~\cite{chakrabarti2012optimal,sherstov2012communication}.

The exact gap-hamming function is defined 
by $$\egh_{n,k}(x,y) = \begin{cases}
1 & \ip{x}{y} = k, \\
0 & \ip{x}{y} = - k ,\\
\star & \text{otherwise.}
\end{cases}$$
As before, we write 
$\egh_n$ to denote $\egh_{n,\lceil \sqrt{n} \rceil}$. 
The exact gap-hamming function is easier to compute than gap-hamming; the protocol only needs to worry about inputs whose inner product has magnitude \emph{exactly}~$k$.  
Proving a sharp lower bound on the 
randomized communication complexity of $\egh$ was left as an open problem.

One of the difficulties in proving a lower bound
for $\egh$ is the following somewhat surprising  
property: 
{\em for infinitely many values of $n$,  
	the deterministic communication complexity of $\egh_n$ is $2$}.
The reason is that there is a simple deterministic protocol of length $2$ that computes
$\ip{x}{y} \mod 4$ for all $n$. This protocol corresponds to the sets $A,B$ discussed with regards to Theorem~\ref{thm:mainP} above. The players announce 
the parities of their inputs $\tfrac{n- \sum_{j=1}^n x_j}{2} \mod 2$ and $\tfrac{n- \sum_{j=1}^n y_j}{2} \mod 2$. These bits  determine $\ip{x}{y} \mod 4$. Indeed, flipping a bit in $x$ changes $\tfrac{n- \sum_{j=1}^n x_j}{2} \mod 2$, and changes  $\ip{x}{y}$ by $+2 \mod 4$.
For example, this deterministic protocol computes $\egh_n$
when $\sqrt{n}$ is an odd integer, because then we have 
$- \sqrt{n} \neq \sqrt{n} \mod 4$.

We overcome this difficulty and show that $\egh$ is extraordinary in that 
although it is a natural problem
with communication complexity $O(1)$ for infinitely many values of $n$, the randomized communication complexity of
$\egh_n$ is at least $\Omega(n)$
for infinitely many values of $n$. Denote by $U_{n,k}$ the uniform distribution over
the set of pairs $(x,y) \in \{\pm 1\}^n \times \{\pm 1\}^n$
so that $\ip{x}{y} \in \{\pm k\}$.

\begin{theorem}\label{thm:LBgen1}
	There is universal constant $\alpha>0$ such that for infinitely many values of~$n$, any protocol that computes $\egh_{n}$
	over inputs from $U_{n,\lceil \sqrt n \rceil}$ with success probability $2/3$ must have communication complexity at least $\alpha n$.
\end{theorem}

There is a natural reduction between different parameters $n,k$, and from randomized protocols to distributional protocols. It turns out that the following theorem is stronger:

\begin{theorem}
	\label{thm:LBgen}
	For every $\beta>0$, there are constants $n_0 > 0$ and $\alpha > 0$ so that the following holds. Let $n,k$ be positive even integers so that $n > n_0$ and
	$k < \alpha \sqrt{n}$. 
	Any protocol that computes $\egh_{n,k}$
	over inputs from $U_{n,k}$ with success probability $2/3$ must have communication complexity at least $(1-\beta) n$.
\end{theorem}

Theorems ~\ref{thm:LBgen1}  and ~\ref{thm:LBgen} are proved in Section ~\ref{sec:egh}. The results are sharp in the following two senses.
First, if $k \neq n \mod 2$ then
$\egh_{n,k}$ is trivial,
and if $k$ is odd then the deterministic communication
complexity of $\egh_{n,k}$ is~$2$.
Secondly, for every $\alpha > 0$, there is $\beta > 0$
so that if $k > \alpha \sqrt{n}$
then the randomized communication complexity of
$\egh_{n,k}$ is at most $(1-\beta)n$. We sketch a randomized protocol for this here.
In the randomized protocol, Alice gets $x$,
Bob gets $y$ and
the public randomness is a sequence
$I_1,I_2,\ldots,I_m$ of i.i.d.\ uniform elements in $[n]$
for $m \leq O(\tfrac{n}{\alpha^2})$. Although $m$ is a constant factor larger than $n$, a standard coupon collector argument shows that 
the number of (distinct) elements in 
the set $S = \{I_1,\ldots,I_m\}$ is at most $(1-\beta)n-1$ with probability
at least~$\tfrac{5}{6}$.
If $|S| > (1-\beta)n-1$, the parties ``abort'',
and otherwise Alice sends to Bob the value of $x_s$ for all $s \in S$.
Bob uses this data to compute
$z = 1+\mathsf{sign} \big( \sum_{j=1}^m x_{I_j}y_{I_j}\big)/2$.
Bob sends the output of the protocol $z$ to Alice.
Chernoff's bound says that if $\egh_{n,k}(x,y) \neq \star$
then $\Pr[z = \egh_{n,k}(x,y)] \geq \tfrac{5}{6}$.
The union bound implies that the overall success probability 
is at least~$\tfrac{2}{3}$.

\subsection*{An application to Additive Combinatorics}

Additive combinatorics studies the behavior
of sets under algebraic operations~\cite{Taovu}.
It has many deep results,
and connections to other areas of mathematics,
as well as many applications in computer science.
Our main result can be interpreted
as showing that Hamming spheres are far from being sum-sets. Our results give quantitative bounds on the size of the intersection of any Hamming sphere with a sum-set.

Replace $\{\pm 1\}$ by the field $\F_2$
with two elements.
The sum-set of $A \subseteq \F_2^n$
and $B \subseteq \F_2^n$ is
$$A+B = \{ x + y : x \in A , y \in B\}.$$
If $X$ and $Y$ are sampled uniformly at random from $A$ and $B$, then $X+Y$ is supported on $A+B$.

The cube $\F_2^n$ is endowed with a natural metric---the Hamming distance $\Delta(x,y)$.
The sphere around $0$ is the collection
of all vectors with a fixed number of ones in them
{(a.k.a.\ a slice)}.
The inner product $I = \sum_j (-1)^{X_j} (-1)^{Y_j}$
is similar to the inner product studied above
(here $X_j,Y_j \in \{0,1\}$).
{The inner product is related to the Hamming distance
by} $I(X,Y) = n- 2 \Delta(X,Y)$. We saw that if $|A| \cdot |B| > 2^{1.01n}$, then $I$ is anti-concentrated.
We can conclude that the distribution of the Hamming
distance of $X+Y$ is anti-concentrated.
The set $A+B$ is far from any slice. In particular, our results imply that for almost all choices of $ a \in A$, we have that $|(a+B) \cap S| \leq O(|B|/\sqrt{n})$ for any slice $S$.


\subsection*{Techniques.}
Chakrabarti and Regev's proof  
uses the deep connection between the
discrete cube and Gaussian space.
They proved a geometric correlation inequality in 
Gaussian space, and then translated it to the cube.
Vidick~\cite{vidick2012concentration}
later simplified part of their argument,
but stayed in the geometric setting.
Sherstov~\cite{sherstov2012communication}
found a third proof
that uses Talagrand's inequality from convex geometry~\cite{talagrand1995concentration}
and ideas of Babai, Frankl and Simon
from communication complexity~\cite{babai1986complexity}.

There are several differences between our argument and the ones in~\cite{chakrabarti2012optimal,vidick2012concentration,sherstov2012communication}.
The main difference is that
the arguments from~\cite{chakrabarti2012optimal,vidick2012concentration,sherstov2012communication} are based, in one way or another,
on the geometry of Euclidean space.
The arguments in~\cite{chakrabarti2012optimal,vidick2012concentration}
prove a correlation inequality in Gaussian space and translate
it to the discrete world. 
It seems that such an argument
can not yield point-wise bounds on the
concentration probability. 
A common ingredient in~\cite{chakrabarti2012optimal,sherstov2012communication}
is a step showing that every set of large enough measure
contains many almost orthogonal vectors
(this is called `identifying the hard core' in~\cite{sherstov2012communication}).
In~\cite{vidick2012concentration} this part of the argument
is replaced by a statement about a relevant matrix.
Our argument does not contain such steps.

Let us briefly discuss our proof at a high level. 
The proof is based on harmonic analysis (Section~\ref{sec:aC}).
The argument consists of two parts.
In the first part, we analyze the Fourier behavior
of $\ip{x}{Y}$
for $x$ fixed and $Y$ random.
We are able to identify a collection
of good $x$'s for which the Fourier spectrum
of the distribution of $\ip{x}{Y}$ decays rapidly.
In the second part,
we show that the number of bad $x$'s
is small by giving an explicit encoding of all of them.

Although the proofs of Theorem ~\ref{thm:mainP} and Theorem ~\ref{thm:main1} follow similar strategies, 
we were not able to completely merge them. 

\section{Harmonic Analysis}
\label{sec:aC}
We are interested in proving
anti-concentration for integer-valued random variables. Harmonic analysis is a natural framework 
for studying such random variables~\cite{halasz1977estimates}.
Let $Y$ be distributed in $\{\pm 1\}^n$.
Let $x \in \Z^n$ be a direction. 
Let $\theta$ be uniformly distributed in $[0,1]$,
independently of $Y$.
The idea is to use
\begin{align*}
\Pr_Y[\ip{x}{Y} = k] &= \Ex{Y}{ \Ex{\theta}{  \exp(2\pi i \theta \cdot (\ip{x}{Y}-k)) }}
\end{align*}
to bound
\begin{align}
\notag
\max_{k \in \Z} \Pr_Y[\ip{x}{Y} = k] \leq  \Ex{\theta}{ \Big |\Ex{Y}{ \exp(2\pi i \theta \cdot\ip{x}{Y})} \Big | } . \tag{$\star$}
\end{align}
This inequality is useful for two reasons.
First, the left hand side is a maximum over $k$,
while the right hand side is not.
So, there is one less quantifier to worry about.
Secondly, the right hand side lives in the Fourier
world, where it is easier to argue about
the underlying operators. 
For example, when the coordinates of $Y$
are independent, the expectation over $Y$
breaks into a product of $n$ simple terms.

\section{The main technical Theorem}
\label{sec:HC}

Our main technical bound is proved in this section. The following theorem controls the Fourier coefficients in most 
directions.

\begin{theorem} \label{thm:tech} 
For every $\beta >0$
and $\delta>0$,
there is $c >0$ so that the following holds.
Let $B \subseteq \nbits^n$ be of size $2^{\beta n}$.
For each $\theta \in [0,1]$, 
for all but $2^{n(1-\beta+\delta)}$ directions
$x \in \{\pm 1\}^n$,
\begin{align*}
\left |\Ex{Y}{ \exp(2\pi i \theta \cdot\ip{x}{Y})} \right | < 
 2 \exp (- c n \sin^2(4 \pi \theta )) 
 \end{align*} 
\end{theorem}

The rest of this section is devoted to proving the theorem.

\subsection{A Single Direction}
\label{sec:fourierbound}
In this section we analyze
the behavior of $\ip{x}{Y}$ for a single direction $x \in \Z^n$.
{We also focus on a single Fourier coefficient $\Ex{Y}{\exp({i \eta \ip{x}{Y}})}$ for a fixed
angle $\eta \in [0,2\pi]$.}

We reveal the entropy of $Y$ coordinate by coordinate.
To keep track of this entropy, define
the following functions $\gamma_1,\ldots,\gamma_n$
from $B = \text{supp}(Y)$ to $\R$.
For each $j \in [n]$, let
$$\gamma_j(y) = \gamma_j(y_{<j}) 
= \min_{\epsilon \in \nbits} \Pr[Y_j = \epsilon | Y_{<j}= y_{<j}].$$

To understand the interaction between $x$
and $y$, we use the following $n$ measurements.
For $j \in [n-1]$,
define $\phi_j(x,y)$ to be half of the phase of the complex number 
$$\Ex{Y_{>j}|Y_j=1,Y_{<j}=y_{<j}}{
\exp({i \eta \ip{x_{>j}}{Y_{>j}}})} \cdot \overline{\Ex{Y_{>j}|Y_j=-1,Y_{<j}=y_{<j}}{\exp({i \eta \ip{x_{>j}}{Y_{>j}}})}}.$$
This quantity is not defined when $\gamma_j(y)=0$.
In this case, set $\phi_j(x,y)$ to be zero. 
Define $\phi_n(x,y)$ to be zero.
The number $\phi_j(x,y)$ is determined by $y_{< j}$ and $x_{>j}$.

In the following we think of $x$ as fixed,
and of $\gamma_j$ and $\phi_j$ as random variables
that are determined by the random variable $Y$.

\begin{lemma} \label{lemma:tech} 
For each $x \in \R^n$,
every random variable $Y$ over $\{\pm 1\}^n$,
and every angle $\eta \in \R$,
{\begin{align*}
	\left |\Ex{Y}{\exp({i \eta \ip{x}{Y}})} \right |^2 & \leq \Ex{Y}{
\prod_{j \in [n]} (1- \gamma_j \sin^2( \phi_j + x_j \eta ))}.
	\end{align*}}
\end{lemma}

\begin{proof}
The proof is by induction on $n$.
We prove the base case of the  induction and the inductive  step 
simultaneously. 
Express
	\begin{align*}
\left |\Ex{Y}{\exp(i \eta \ip{x}{Y})} \right |^2  &= \left|\Ex{Y_1}{ \exp(i \eta  x_1 Y_1) \cdot \Ex{Y_{>1}|Y_1}{\exp(i \eta  \ip{x_{>1}}{Y_{>1}})}}\right|^2 \\
& = \left|p_1 \exp(i \eta x_1) Z_1 + p_{-1} \exp(- i \eta  x_1 ) Z_{-1}\right|^2,
\end{align*}
where for $\epsilon \in \{\pm 1\}$, $$p_\epsilon = \Pr[Y_1 = \epsilon]
\qquad \& \qquad Z_{\epsilon} = \Ex{Y|Y_1 =\epsilon}{\exp(i \eta \ip{x_{>1}}{Y_{>1}})}.$$
When $n=1$, we have $Z_1=Z_{-1}=1$.	 
Rearranging,
\begin{align*} 
&\left|p_1 \exp(i \eta  x_1) Z_1 + p_{-1} \exp(-i \eta  x_1 ) Z_{-1}\right|^2 \\ 
& = p_1^2 |Z_1|^2 + p_{-1}^2 |Z_{-1}|^2 + p_1 p_{-1} (Z_1 \overline{Z_{-1}} \exp(i2\eta  x_1)+ \overline{Z_1} Z_{-1} \exp(-i2\eta x_1)) \\
&= p_1^2 |Z_1|^2 + p_{-1}^2 |Z_{-1}|^2 + 2 p_1 p_{-1} |Z_1| |Z_{-1}|  \cos(2 \phi_1 + 2 x_1 \eta) .
\end{align*} 
The last equality holds by the definition of $\phi_1$.

{There are two cases to consider.
When $\cos(2 \phi_1 + 2 x_1 \eta) < 0$, we continue to bound
\begin{align*} 
& < p_1^2 |Z_1|^2 + p_{-1}^2 |Z_{-1}|^2  \\
& \leq (p_1 |Z_1|^2 + p_{-1} |Z_{-1}|^2)(1-\gamma_1) \\
& \leq (p_1 |Z_1|^2 + p_{-1} |Z_{-1}|^2)(1-\gamma_1 \sin^2(\phi_1 +x_1 \eta)) .
\end{align*} 
Recall that $\gamma_1$ and $\phi_1$
do not depend on $Y$.
When $\cos(2 \phi_1 + 2 x_1 \eta) \geq 0$,
using the inequality $a^2 + b^2 \geq 2ab$, we bound}
\begin{align*}
&\leq p_1^2 |Z_1|^2 + p_{-1}^2 |Z_{-1}|^2 + p_1 p_{-1}(|Z_1|^2+  |Z_{-1}|^2)  \cos(2\phi_1 + 2x_1 \eta) \\
&= p_1 |Z_1|^2 (p_1 + p_{-1} \cos(2\phi + 2x_1 \eta)) + p_{-1} |Z_{-1}|^2 (p_{-1} + p_{1} \cos(2\phi_1 + 2x_1 \eta)) \\
& \leq (p_1 |Z_1|^2 + p_{-1} |Z_{-1}|^2)   (1-\gamma_1 + \gamma_1 \cos(2 \phi_1 + 2x_1 \eta)) \\
& = \Ex{Y_1}{|Z_{Y_1}|^2}  (1 -2 \gamma_1  \sin^2(\phi_1 + x_1 \eta)) .
\end{align*}
When $n=1$, we have proved the base case of the induction.
When $n>1$, apply induction
on $|Z_\eps|^2$.
\end{proof}

\subsection{A Few Bad Directions} 
\label{sec:encoding}

Lemma~\ref{lemma:tech} suggests 
proving that the expression 
$$\sum_j \gamma_j \sin^2( \phi_j + x_j \eta )$$
is typically large. Namely, we aim to show that there are usually many coordinates $j$ for which 
both $\gamma_j$ and $\sin^2(\phi_j +  x_j \eta )$ are bounded away from zero.
{Our approach is to explicitly encode the cases
where this fails to hold.}

Recall that $Y$ is uniformly distributed in a set $B$ of size $|B| = 2^{\beta n}$.
Let $1\geq \lambda>1/n$ be a parameter.  
Set $0 < \kappa < \tfrac{1}{2}$ and $1\geq \tau>0$ to be parameters satisfying the conditions
\begin{align}
\label{1}
H\left(\tfrac{1}{\log(1/\kappa)}\right) =  \tau + H\left(\tau \right)= \lambda , 
\end{align}
where $H$ is the binary entropy function: $$H(\xi) = \xi \log(1/\xi) + (1-\xi) \log(1/(1-\xi)).$$

The encoding is based on the following two sets:
$$J(y)  = J_{B,\kappa}(y) = \{ j \in [n] : \gamma_j(y) \geq \kappa\}$$ 
and
$$G(x,y) = G_{B,\kappa,\theta}(x,y) = \Big\{j \in J(y) : \sin^2(\phi_j(x,y) + x_j \eta ) \geq \tfrac{\sin^2(2\eta)}{4}\Big\}.$$
We start by showing that there 
are few $y$'s for which $|J(y)|$ is small.

\begin{lemma}\label{lemma:encoding1}
The number of $y \in B$ with $|J(y)| \leq n(\beta -3\lambda)$ is at most $2^{n(\beta - 2\lambda)}$.
\end{lemma}
\begin{proof}
If $3 \lambda > \beta$, 
the statement is trivially true. So, in the rest of the proof, assume that $3\lambda \leq \beta$.
Each  $y \in B$ with $|J(y)| \leq  n(\beta - 3 \lambda)$  can be uniquely  encoded by the following data:
\begin{itemize}
\item[--] A vector $q \in \nbits^{t}$ with
$t = \lfloor n(\beta - 3 \lambda) \rfloor$.
\item[--] A subset $S \subseteq [n]$ of size $|S| \leq \tfrac{n}{ \log(1/\kappa)}$.
\end{itemize} 
Let us describe the encoding. 
The vector $q$ encodes 
the values taken by $y$ in the coordinates $J(y)$. We do not encode $J(y)$ itself, only the values of $y$ in the coordinates corresponding to $J(y)$.  
The set $S$ includes $j \in [n]$ if and only if 
$$\Pr[Y_j = y_j|Y_{<j} = y_{<j}] < \kappa.$$ Each string $y \in B$ has probability at least $2^{-n}$.
This implies that
$\kappa^{|S|} \geq 2^{-n}$. 

We can reconstruct $y$ from $q$ and $S$ 
by iteratively computing $y_1$, then $y_2$, and so on, until we 
get to $y_n$. Whether or not $1 \in J(y)$ is determined even before we know $y$. If $1 \in J(y)$ then $q$ tells us
what $y_1$ is.
If $1 \not \in J(y)$ and $1 \in S$ then
$y_1$ is the least likely value between $\pm 1$.
If $1 \not \in J(y)$ and $1 \not \in S$
then $y_1$ is the more likely value.
Given the value of $y_1$,
we can continue in the same way to compute the rest of $y$.

The number of choices for $q$ is at most $2^{n(\beta - 3\lambda)}$. 
The number of choices for $S$ is at most 
$2^{n H(1/\log(1/\kappa))} = 2^{\lambda n}$.
\end{proof}

Next, we argue that there 
are few $x$'s for which there are many $y$'s with small $G(x,y)$.

\begin{lemma} \label{lemma:encoding2}
 The number of $x \in A$ for which $$\Pr_Y[ |G(x,Y)| \leq \tau n ] \geq 2^{-\lambda n}$$ is at most $2^{n(1-\beta+6\lambda)}$. 
\end{lemma}

\begin{proof}
The lemma is proved by double-counting
the edges in a bipartite graph.
Let ${\calX}$ be the set we are interested in:
	$$\calX = \big\{x : \Pr_Y[ |G(x,Y)| \leq \tau n ] \geq 2^{-\lambda n}\big\}.$$
The left side of the bipartite graph is $\calX$
and the right side is $B$.
Connect $x \in \calX$ to $y \in B$ by an edge
if and only if $G(x,y) \leq \tau n$.
Let  $E$ denote the set of edges in this graph.

First, we bound the number of edges from below.
The number of edges that
touch each $x \in \calX$ is at least
$2^{-\lambda n} |B|$.
It follows that
$$|E| \geq 2^{-\lambda n} \cdot |\calX| \cdot |B|.$$

Next, we bound the number of edges from above.
By Lemma~\ref{lemma:encoding1},
the number of 
$y \in B$ so that $|J(y)| \leq n(\beta -3\lambda)$
is at most $2^{-2\lambda n}|B|$.
We shall prove that the number of edges
that touch each $y$ with  $|J(y)| > n(\beta -3\lambda)$ is at most
$2^{n(1-\beta + 4 \lambda)}$.
It follows that 
$$|E| \leq 2^{-2\lambda n }\cdot |\calX|\cdot |B|
+ |B|\cdot 2^{n(1-\beta + 4 \lambda)}.$$
We can conclude that
\begin{align*}
2^{-\lambda n}\cdot  |\calX| \cdot |B| &\leq 2^{-2\lambda n}\cdot |\calX| \cdot |B|
+ |B| \cdot 2^{n(1-\beta + 4 \lambda)} \\
\Rightarrow |\calX| & \
\leq 2^{n(1-\beta + 6 \lambda)},
\end{align*}
since $\lambda n >1$.

It remains to fix $y$ so that
$|J(y)| > n(\beta -3\lambda)$ and bound its degree from above.
This too is achieved by an encoding argument.
Encode each $x$ that is connected
to $y$ by an edge using the following data:
	\begin{itemize}
		\item[--] A vector $q \in \nbits^{t}$
		with $t = \lfloor n(1-\beta+3\lambda) \rfloor$.
		\item[--] The set $G(x,y)$.
		\item[--] A vector $r \in \nbits^{s}$ with
		$s = \lfloor \tau n \rfloor$.
	\end{itemize}
Let us describe the encoding. 
The vector $q$ specifies the values of $x$ on coordinates not in $J(y)$.
There are at most $n- n(\beta -3\lambda) = n(1-\beta + 3\lambda)$ such coordinates. 
The size of $G(x,y)$ is at most $\tau n$.
The vector $r$ specifies the values of $x$ in the coordinates of $G(x,y)$, written in descending order. 
    
    The decoding of $x$ from $q,S$ and $r$ is done as follows.
    Decode the coordinates of $x$ in descending order from $n$ to $1$. 
    If $n \not \in J(y)$ then we read the value of $x_n$ from $q$. 
    If $n \in J(y)$ and $n  \in G(x,y)$, we decode $x_n$ by reading its value from $r$. If $n \in J(y)$ and 
   $n \notin G(x,y)$, 
  then 
  $$\sin^2(\phi_n(x,y)  + x_n \eta) \leq \tfrac{\sin^2(2\eta )}{4}.$$ 
   The number $\phi_n(x,y)$ does not depend on $x$.
   The following claim implies that there is at most one value of $x_n$ that satisfies this property.
   
   \begin{claim}
   \label{clm:OneGoodChoice}
   For all $\varphi \in \R$ and $u,v \in \Z$,
   \begin{align*}
   \max\{ |\sin(\varphi +\eta u)| , | \sin(\varphi + \eta v)| \}
\geq \tfrac{|\sin(\eta (u - v))|}{2}. 
   \end{align*}
\end{claim}

\begin{proof}
We wish to show that the two points on the unit circle of phase $\varphi +\eta u$ and $\varphi +\eta v$ cannot both be very close to the real line in general.  	
Consider the map 
$$\varphi\mapsto g(\varphi) =  
\max \{|\sin(\varphi +\eta u)| , |\sin(\varphi + \eta v)|\}.$$
Observe that the minimum of this map 
is attained when $$|\sin(\varphi +\eta u)| = |\sin(\varphi + \eta v)|,$$ since if this is not the case, we can change $\varphi$ by a little to reduce the larger of the two magnitudes. Now, $|\sin(\alpha)| = |\sin(\beta)|$ when $\alpha + \beta$ is an integer multiple of $\pi/2$. Thus, 
the two magnitudes are equal exactly when  
 $\varphi = -\tfrac{\eta (u + v)}{2} + t \pi/2$, for some integer $t$. By symmetry, it is enough to consider $t \in \{0,1\}$, so we obtain that 
\begin{align*}
g(\varphi) &\geq \min \{g(-\eta (u + v)/2),g(-\eta (u + v)/2+\pi/2)\} \\
&\geq |g(-\eta (u + v)/2) \cdot g(-\eta (u + v)/2+\pi/2)|\\
&\geq
|\sin(\eta (u-v)/2) \cdot \cos (\eta (u - v)/2)|
\\
&= \frac{|\sin(\eta (u-v))|}{2} . \qedhere
\end{align*}
\end{proof}

The claim implies that we can indeed reconstruct $x_n$.
Given $x_n$, we can similarly reconstruct
$x_{n-1}$, since $\phi_{n-1}$ depends only on $y$ and $x_n$. Continuing in this way, we can reconstruct $x_{n-2},\dotsc,x_1$. The total number of choices for $q,S,r$ is at most 
$2^{n(1-\beta+3\lambda) + n H(\tau) + \tau n  }
= 2^{n(1-\beta + 4 \lambda)}$.\qedhere

    
\end{proof}

%
%
%
%
%
%

\begin{proof}[Proof of Theorem \ref{thm:tech}]
	Set $\lambda = \tfrac{\delta}{6}$.
	By Lemma \ref{lemma:tech}, 
	\begin{align*}
	\left| \Ex{Y}{\exp(2 \pi i \theta \ip{x}{Y})} \right| &
	\leq 
	\sqrt{\Ex{Y}{\exp\left (- \sum_{j=1}^n \gamma_j \sin^2(\phi_j + 2 \pi \theta  x_j )\right)}.}
	\end{align*}
	Whenever $x$ is such that 
	\begin{align}
	\label{2}
	\Pr_Y[ G(x,Y) \leq \tau n ] < 2^{-\lambda n},
	\end{align}
	we can bound 
	$$ \Ex{Y}{\exp\left (- \sum_{j=1}^n \gamma_j \sin^2(\phi_j + 2 \pi \theta  x_j )\right)}
	\leq \exp  (- \tfrac{ \kappa}{4} n \tau \sin^2(4 \pi \theta))  
	+ 2^{-\lambda n}.$$
	Since
	$\sqrt{a+b} \leq \sqrt{a}+\sqrt{b}$ for $a,b \geq 0$,
	for such an $x$ we can bound
	\begin{align*}
	\left| \Ex{Y}{\exp(2 \pi i \theta \ip{x}{Y})} \right| &
	\leq  \exp  (- \tfrac{ \kappa}{8} n \tau \sin^2(4 \pi \theta))  
	+ 2^{-\lambda n/2} \\
	& \leq 2 \exp ( - c n \sin^2(4 \pi \theta) ) .
	\end{align*}
	Lemma~\ref{lemma:encoding2} promises that there are at most $2^{n(1-\beta+ \delta)}$ choices for $x$ that does not satisfy~\eqref{2}.
	
\end{proof}

\section{Smoothness}\label{sec:smooth}
To prove smoothness, we use Theorem ~\ref{thm:tech}. The constant $4 \pi$ on the r.h.s.\ of the bound in the theorem corresponds to a step size of $4$.

\begin{proof}[Proof of Theorem~\ref{thm:mainP}]
	 Theorem~\ref{thm:tech} with $\delta = \tfrac{\eps}{2}$ promises that 
	for each $\theta \in [0,1]$, the size of
	\begin{align*}
	A_\theta = \Big\{x \in \{\pm 1\}^n :
	\Big |\Ex{Y}{ \exp(2\pi i \theta \ip{x}{Y})} \Big | > 
	2\exp (-c n \sin^2(4 \pi \theta ))  \Big\}
	\end{align*} 
	is at most $2^{n(1-\beta+\delta)}$.
	For each $x$, define 
	$S_x = \{\theta \in [0,1]: x \in A_\theta\}$.
	
	Fix $x$ such that $|S_x| \leq 2^{-\delta n}$.
	Bound
	\begin{align*}
	& \big| \Pr_Y[\ip{x}{Y} = k]
	- \Pr_Y[\ip{x}{Y} = k+4] \big| \\
	&= \Big| \E_Y \Big[ \int_{0}^{1} 
	\exp(2\pi i \theta (\ip{x}{Y}-k)) - \exp(2\pi i \theta (\ip{x}{Y}-k-4))
	\, \mathrm{d}\theta \Big]\Big| \\
	&\leq  \int_{0}^{1} |\exp(4 \pi i \theta) - \exp(-4 \pi i \theta)| \cdot 
	\Big| 
	\Ex{Y}{ \exp(2\pi i \theta \ip{x}{Y})} \Big| \, \mathrm{d}\theta \\
	&\leq  2 \int_{0}^{1} |\sin(4 \pi \theta)| \cdot 
	\Big| 
	\Ex{Y}{ \exp(2\pi i \theta \ip{x}{Y})} \Big| \, \mathrm{d}\theta .
	\end{align*}
	Continue to bound
	\begin{align*}
	& \int_{0}^{1} 
	|\sin(4 \pi \theta)| \cdot 
	\Big| \Ex{Y}{ \exp(2\pi i \theta \ip{x}{Y})} \Big| \, \mathrm{d}\theta \\
	& \leq  2^{-\delta n}+ 2 \int_{0}^{1} |\sin(4 \pi \theta)| \cdot 
	\exp(- c n \sin^2(4 \pi\theta)) \, \mathrm{d}\theta .
	\end{align*}
	The integral goes around the circle twice, and it is identical in each quadrant. So,
	\begin{align*}
	&\int_{0}^{1} |\sin(4 \pi \theta)| \cdot 
	\exp(- c n \sin^2(4 \pi\theta)) \, \mathrm{d}\theta \\
	&= 8 \int_{0}^{1/8} \sin(4 \pi \theta) \cdot 
	\exp(- c n \sin^2(4 \pi\theta)) \, \mathrm{d}\theta \\
	&\leq 32 \pi \int_{0}^{\infty} \theta \cdot 
	\exp(- 16 c n  \theta^2) \, \mathrm{d}\theta \\
	&\leq \tfrac{c_1}{n}  \int_{0}^{\infty} \phi \cdot 
	\exp(-  \phi^2) \, \mathrm{d}\phi \leq \tfrac{C}{n},
	\end{align*}
	where $c_1, C > 0$ depend on $\eps$, and we used $\tfrac{\eta}{\pi} \leq \sin(\eta) \leq \eta$ for $0 \leq \eta \leq \tfrac{\pi}{2}$.
	
	Finally, because 
	$$\E_x |S_x| = \E_\theta \tfrac{|A_\theta|}{2^n} 
	\leq 2^{n(-\beta+\delta)},$$  
	by Markov's inequality,
	the number of $x \in \{\pm 1\}^n$ for which $|S_x| > 2^{-\delta n}$ is at most 
	$2^{(1-\beta+2\delta) n} = 2^{(1-\beta+\epsilon) n}$. 
\end{proof}

\section{Anti-concentration in General Two-Cubes}
\label{sec:UTC}
Now we move to the setting
where the direction $x$ is chosen from 
an arbitrary two-cube $A \subset \Z^n$
with differences $d_1,\ldots,d_n$;
our goal is to prove Theorem~\ref{thm:main1}.
The way we measure the structure
of $A$ follows ideas of Hal{\'a}sz~\cite{halasz1977estimates}.
For an integer $\ell > 0$, define $r_\ell(A)$ to be
the number of elements $(\epsilon, j) \in \{\pm 1\}^{2\ell} \times [n]^{2\ell}$ that satisfy $$\epsilon_{1} \cdot d_{j_1} + \dotsb + \epsilon_{2\ell} \cdot d_{j_{2\ell}} = 0.$$
The smaller $r_\ell(A)$ is, the less structured $A$ is.

The theorem below shows that $r_\ell(A)$ allows us 
to control the concentration probability. More concretely, for $C>0$ and $\ell > 0$, define 
$$R_{C,\ell}(A) = 
 \frac{C^\ell r_\ell(A)}{n^{2\ell+ 1/2}} + \exp(-\tfrac{n}{C} ).$$
Define
$$R_{C}(A) = 
\inf \{  R_{C,\ell}(A) : \ell \in \N \}.$$
This is essentially the bound on the concentration
probability that Hal{\'a}sz obtained in~\cite{halasz1977estimates}
when $Y$ is uniform in $\{\pm 1\}^n$.
Our upper bounds are slightly weaker. Let 
$$\mu_C(A)
= \inf \Big\{
\mu  \in [0,1] : \exists
 \nu \in (0,1] \ \ 
   \mu^{(1+\nu)^2}
  \geq 3  \exp(-\tfrac{ \nu n}{C}) +
 \tfrac{R_{C}(A)}{50 \sqrt{\nu}} \Big\},$$
where we adopt the convention that the infimum of the empty set is~$1$. 
 Before stating the theorem, let us go over the three
 examples from Theorem~\ref{thm:main1}:
\begin{enumerate}
\item
For arbitrary $A$, 
since $r_1(A) \leq O(n^2)$, we get\footnote{Here
and below the big $O$ notation
hides a constant that may depend on $C$.}
$\mu_C(A) \leq O(\tfrac{\sqrt{\ln n}}{\sqrt{n}})$
with $\nu = \tfrac{1}{\ln(1/R_{C,1}(A))}$.
\item When all the differences are distinct,
since $r_1(A) \leq O(n)$, we get
$\mu_C(A) \leq O(n^{-1.5} \sqrt{\ln n})$
with $\nu = \tfrac{1}{\ln(1/R_{C,1}(A))}$.
\item When $\{\pm d_1,\ldots,\pm d_n\}$ is a Sidon set,
since $r_2(A) \leq O(n^2)$,
we get
$\mu_C(A) \leq O(n^{-2.5} \sqrt{\ln n})$
with $\nu = \tfrac{1}{\ln(1/R_{C,2}(A))}$.
\end{enumerate}
More generally, when $R_C(A)$ is bound
from below by some
polynomial in $\tfrac{1}{n}$
then $\mu_C(A)$ is at most $O(R_C(A) \sqrt{ \log (4/R_C(A))})$.


%

\begin{theorem}
\label{thm:mainTechU}
For every $\beta > 0$ and $\delta >0$, 
there is $C>0$ so that the following holds.
Let $B \subseteq \{\pm 1\}^n$ be of size $2^{\beta n}$.
Let $Y$ be uniformly distributed in $B$.
Let $A \subset \Z^n$ be a two-cube.
Then, for all but $2^{n(1-\beta+\delta)}$ directions $x \in A$,
\begin{align*}
\Ex{\theta}{ \left |\Ex{Y}{ \exp(2\pi i \theta \cdot\ip{x}{Y})} \right|}
\leq \mu_C(A) .
\end{align*}
\end{theorem}

Before moving on, we discuss a fourth
extreme example. 
When $A_j = \{2^j,-2^j\}$ for each $j \in [n]$,
we have $r_\ell(A) \leq (2\ell n)^\ell$.
In this case, setting $\ell = \Omega(n)$ gives exponentially small anti-concentration with $\nu=1$.
This result is trivial, but it illustrates that the mechanism
underlying the proof yields strong bounds in
many settings.

By ($\star$) from Section~\ref{sec:aC} and the explanation above,
we see that
Theorem~\ref{thm:mainTechU} implies Theorem~\ref{thm:main1}.
The rest of this section is devoted to the proof
of Theorem~\ref{thm:mainTechU}.
The high-level structure of the proof is similar
to that of Theorem~\ref{thm:tech}.
However, there are several new technical challenges
that we need to overcome. 

The main technical challenge that needs to be overcome has to do with the definition of the set $G$.  
The $G$ defined in the previous section depends on the angle $\theta$.
This is problematic for the proof in the generality 
we are working with now.
So, we need to find a different set of \emph{good} coordinates, one that depends only on $x$ and $y$.
Our solution is based on 
the following claim, which quantifies the strict
convexity of the map $\zeta \mapsto \zeta^{1+\nu}$
for $\nu >0$. We defer the proof to Appendix~\ref{sec:techclaim}.

\begin{claim}\label{clm:MinPhi} 
For every $\kappa > 0$,
there is a constant $c_1 > 0$ so that the following holds.
For every random variable $W \in \nbits$ such that
$$\min \big\{\Pr[W=1], \Pr[W=-1]\big\} \geq \kappa,$$ 
every $\alpha_1
\geq 2 \alpha_{-1}\geq 0$ 
and every $0 < \nu \leq 1$, 
	$$\Expect{\alpha_W}^{1+\nu} \leq (1-c_1 \nu)  \Expect{\alpha_W^{1+\nu}}.$$
\end{claim}

%

\subsection{A Single Direction}
\label{sec:fourierboundU}

The following lemma generalizes Lemma~\ref{lemma:tech}.
Recall the definition of $\gamma_j$, $\phi_j$
and $J(y)$ from Sections~\ref{sec:fourierbound}
and~\ref{sec:encoding}.

\begin{lemma} \label{lemma:techU} 
For every $\kappa > 0$,  
there is a constant $c_0 > 0$ so that the following holds. For every $0 < \nu\leq 1$, every angle $\eta \in \R$, every direction $x \in \Z^n$, and every random variable $Y$ over $\{\pm 1\}^n$,
\begin{align*}
	\left |\Ex{Y}{\exp({i \eta \ip{x}{Y}})} \right|^{1+\nu} & \leq \Ex{Y}{
	\prod_{j \in J} (1- c_0 \nu \sin^2( \phi_j + x_j \eta ) }.
\end{align*}
\end{lemma}

\begin{proof}
The proof is by induction on $n$. If $1 \notin J$, the proof holds by induction. The base case of $n=1$ is trivial.
So assume that $1 \in J$. 
Express
\begin{align*}
\Ex{Y}{\exp(i \eta \ip{x}{Y})}  
& =  p_1 \exp(i \eta x_1) Z_1 + p_{-1} \exp(- i \eta  x_1 ) Z_{-1} ,
\end{align*}
where for $\epsilon \in \{\pm 1\}$, $$p_\epsilon = \Pr[Y_1 = \epsilon]
\qquad \& \qquad Z_{\epsilon} = \Ex{Y|Y_1 =\epsilon}{\exp(i \eta \ip{x_{>1}}{Y_{>1}})}.$$
When $n=1$, we have $Z_1=Z_{-1}=1$.	 
Using the definition of $\phi_1$,
\begin{align*} 
&\left|p_1 \exp(i \eta  x_1) Z_1 + p_{-1} \exp(-i \eta  x_1 ) Z_{-1}\right|^2 \\ 
& = p_1^2 |Z_1|^2 + p_{-1}^2 |Z_{-1}|^2 + p_1 p_{-1} (Z_1 \overline{Z_{-1}} \exp(i2\eta  x_1)+ \overline{Z_1} Z_{-1} \exp(-i2\eta x_1)) \\
&= p_1^2 |Z_1|^2 + p_{-1}^2 |Z_{-1}|^2 + 2 p_1 p_{-1} |Z_1| |Z_{-1}|  \cos(2 \phi_1 + 2 x_1 \eta) \\
&= p_1^2 |Z_1|^2 + p_{-1}^2 |Z_{-1}|^2 + 2 p_1 p_{-1} |Z_1| |Z_{-1}| 
\\ 
& \qquad -2 p_1 p_{-1}  |Z_1| |Z_{-1}| (1-\cos(2 \phi_1 + 2 x_1 \eta)) \\
&= \Expect{|Z_{Y_1}|}^2 - 
4 p_1 p_{-1} |Z_1| |Z_{-1}| \sin^2 ( \phi_1 +  x_1 \eta),
\end{align*} 
Without loss of generality, assume that $|Z_1| \geq |Z_{-1}|$. There are two cases to consider.
The first case is that $Z_1$ and $Z_{-1}$ are comparable in  magnitude: $|Z_1| \leq 2 |Z_{-1}|$. In this case,  we can continue the bound by
\begin{align*} 
& \leq \Expect{|Z_{Y_1}|}^2 - 
2 p_1 p_{-1} |Z_1|^2 \sin^2 ( \phi_1 +  x_1 \eta) \\
& \leq \Expect{|Z_{Y_1}|}^2 (1 - 
2 \kappa (1-\kappa)  \sin^2 ( \phi_1 +  x_1 \eta) ) ,
\end{align*} 
since $1 \in J$. This gives
\begin{align*} 
&\left|p_1 \exp(i \eta  x_1) Z_1 + p_{-1} \exp(-i \eta  x_1 ) Z_{-1}\right|^{1+\nu} \\ 
& \leq \Expect{|Z_{Y_1}|}^{1+\nu}(1 - 
2 \kappa (1-\kappa)  \sin^2 ( \phi_1 +  x_1 \eta) )^{(1+\nu)/2}\\
& \leq \Expect{|Z_{Y_1}|^{1+\nu}}(1 - 
 \kappa (1-\kappa)  \sin^2 ( \phi_1 +  x_1 \eta) ),
\end{align*} 
since the map $\zeta \mapsto \zeta^{1+\nu}$ is convex.

The second case is when $|Z_1| > 2|Z_{-1}|$.
Recall that we have already shown
\begin{align*} 
&\left|p_1 \exp(i \eta  x_1) Z_1 + p_{-1} \exp(-i \eta  x_1 ) Z_{-1}\right|^2 \\ 
&= \Expect{|Z_{Y_1}|}^2 - 
4 p_1 p_{-1} |Z_1| |Z_{-1}| \sin^2 ( \phi_1 +  x_1 \eta) \\
& \leq \Expect{|Z_{Y_1}|}^2 . 
\end{align*}
Claim \ref{clm:MinPhi} implies that 
\begin{align*} 
&\left|p_1 \exp(i \eta  x_1) Z_1 + p_{-1} \exp(-i \eta  x_1 ) Z_{-1}\right|^{1+\nu} \\ 
& \leq \Expect{|Z_{Y_1}|}^{1+\nu} \\
& \leq (1-c_1 \nu)  \cdot \Expect{|Z_{Y_1}|^{1+\nu}} \\
& \leq (1-c_1 \nu \sin^2( \phi_j + x_j \eta ) )   \cdot \Expect{|Z_{Y_1}|^{1+\nu}}.
\end{align*} 

Finally, setting $c_0 = \min\{c_1, \kappa (1-\kappa)\}$, we get a bound that applies in both cases:
\begin{align*} 
\left |\Ex{Y}{\exp({i \eta \ip{x}{Y}})} \right|^{1+\nu}
& \leq (1-c_0 \nu \sin^2( \phi_j + x_j \eta ) ) \cdot   \Expect{|Z_{Y_1}|^{1+\nu}}.
\end{align*}
This proves the base case of the induction and 
also allows to perform the inductive step.

\end{proof}

\subsection{An Average Direction} 
\label{sec:fourierboundAverageU}
In this section we analyze
the bound from the previous section for an average direction $X$
in a two-cube $A \subset \Z^n$.
This step has no analogy in the proof
of Theorem~\ref{thm:tech}.
To compute the expectation over an average direction, we reveal the entropy of $X$ coordinate by coordinate
in reverse order (from the $n$'th coordinate to the first one).

In analogy with $\gamma_1, \dotsc, \gamma_n$, define
the following functions $\mu_1,\ldots,\mu_n$.
For each $j \in [n]$, let
$$\mu_j(x) = \mu_j(x_{>j}) 
= \min_{\epsilon \in A_j} \Pr[X_j = \epsilon | X_{>j}= x_{>j}];$$
this is well-defined for $x$ in $A = \text{supp}(X)$.
In analogy with the definition of $J(y)$,
let 
$$J'(x) = \{j \in [n]: \mu_j(x) \geq \kappa\}.$$
In this section, we define the set $G$ differently, but use the same notation. Let 
$$ G(x,y) = G_{A,B,\kappa}(x,y) =  J'(x) \cap J(y).$$
Recall that $\gamma_j$, $\phi_j$
and $J(\cdot)$ depend on the set $B$,
on $y \in B$ and on $x \in \Z^n$.
In the following lemma, we fix an arbitrary $y \in B$, and take the expectation over a random $X \in A$.
We allow $G$ to be a random set that depends on $X$, and
$\phi_j$ to be a random variable
that depends on $X_{>j}$.

\begin{lemma} \label{lemma:techU2} 
For every $\kappa > 0$ and $0 < c_0 \leq 1$,  
		there is a constant $c > 0$ so that the following holds. For every $0 < \nu\leq 1$, every angle $\eta \in \R$, 
		every $B \subseteq \{\pm 1\}^n$,
		every $y \in B$,
every random variable $X$ taking values in a two-cube $A\subseteq \Z^n$ with differences $d_j = u_j - v_j$, 
		\begin{align*}
		\Ex{X}{\prod_{j \in J} (1-c_0 \nu \sin^2(\phi_j + X_j \eta))}^{1+\nu} & \leq \Ex{X}{\exp\Big(- c \nu \sum_{j \in G} \sin^2(d_j \eta) \Big)}.
		\end{align*}
\end{lemma}

\begin{proof}
The proof is by induction on $n$. 
Recall that
$\phi_j$ and $\mu_j$ is determined by $x_{>j}$.
In particular, whether or not $n \in G(x,y)$ 
does not depend on $x$.
If $n \notin G(x,y)$, the proof holds by induction, or is trivially true for $n=1$. 
So assume that $n \in G(x)$. 
Start with
\begin{align*}
 	& \Ex{X}{\prod_{j \in J} (1- c_0 \zeta \sin^2( \phi_j + X_j \eta ))} \\
	 	& = \Ex{X_{n}}{(1- c_0 \zeta \sin^2( \phi_n + X_n \eta )) Z_{X_n}  },	\
		\end{align*}
where for $a \in A_n : = \{u_n,v_n\}$,
$$Z_a = \Ex{X|X_{n}=a}{
		\prod_{j \in J:j<n} (1- c_0 \sin^2( \phi_j + X_j \eta ))}.$$
		If $n=1$, then $Z_u = Z_v = 1$.
Assume without loss of generality 
		that $Z_u \geq Z_v$.
There are two cases to consider.
The first case is that $Z_u > 2 Z_v$.
In this case, Claim~\ref{clm:MinPhi} implies
\begin{align*}
 	 \Ex{X}{\prod_{j \in J} (1- c_0 \nu \sin^2( \phi_j + X_j \eta ))}^{1+\nu}  &\leq
		\Expect{Z_{X_n}}^{1+\nu} \\
		&\leq (1- c_1 \nu) \Expect{Z_{X_n}^{1+\nu}}\\
		&\leq \exp(- c_1 \nu) \Expect{Z_{X_n}^{1+\nu}}.
		\end{align*} 
The second case is when $Z_u \leq 2 Z_v$. 
By Claim~\ref{clm:OneGoodChoice}, 
$$\max \big\{ |\sin(\phi_n + u \eta)| , |\sin(\phi_n + v \eta)|
\big\} \geq \tfrac{\sin(d_n \eta)}{2}.$$ 
Since $\mu_n(x) \geq \kappa$,
\begin{align*}
 &\Ex{X_{n}}{(1- c_0 \nu \sin^2( \phi_n + X_n \eta )) Z_{X_n}  }^{1+\nu} \\
 & \leq (\Ex{X_{n}}{Z_{X_n}  } - \kappa c_0 \nu \tfrac{\sin^2(d_n \eta)}{4} \tfrac{Z_{u}}{2})^{1+\nu}\\
 & \leq (\Ex{X_{n}}{Z_{X_n}} (1 -  \tfrac{\kappa  c_0 \nu}{8} \sin^2(d_n \eta)))^{1+\nu}\\
 & \leq \Ex{X_{n}}{Z^{1+\nu}_{X_n}}  \exp( - \tfrac{c_0 \kappa \nu}{8} \sin^2(d_n \eta)).
\end{align*}

In both cases, 
$$ \Ex{X}{\prod_{j \in J} (1- c_0 \sin^2( \phi_j + X_j \eta ))}^{1+\nu} \leq \exp(-c \nu \sin^2(d_n \eta))  \Ex{X_n}{Z_{X_n}^{1+\nu}},$$ for some constant $c(\kappa,c_0) >0$.
This proves the base case of the induction and 
also allows to perform the inductive step.

\end{proof}

\subsection{Putting It Together}
\label{sec:mainPfU}

\begin{proof}[Proof of Theorem \ref{thm:mainTechU}]
Let $\mu >0$ and $0 < \nu \leq 1$
be so that 
  $$ \mu^{(1+\nu)^2}
  \geq 3  \exp(-\tfrac{ \nu n}{C}) +
 \tfrac{R_C(A)}{50 \sqrt{\nu}};$$
if no such $\mu,\nu$ exist then the theorem is trivially true.
Let $$A_0 = \Big\{ x \in A : 
\Ex{\theta}{ \Big |\Ex{Y}{ \exp(2\pi i \theta \cdot\ip{x}{Y})} \Big|}
 \geq \mu \Big\}.$$
Denote the size of $A_0$ by $2^{\alpha n}$.
Assume towards a contradiction that 
$\alpha +\beta \geq 1+\delta$.
Let $X$ be uniformly distributed in $A_0$, independently of $Y$ and $\theta$.
Let $\lambda = \tfrac{\delta}{7}$,
and let $\kappa$ be as in~\eqref{1}.
By Lemma~\ref{lemma:techU},
		\begin{align*}
& \Ex{X,\theta}{\Big|\Ex{Y}{\exp({i 2\pi \theta \ip{x}{Y}})}\Big|}^{(1+\nu)^2} \\
& \leq \Ex{X,\theta}{\Big|\Ex{Y}{\exp({i 2\pi \theta \ip{x}{Y}})}\Big|^{1+\nu}}^{1+\nu} \\
& \leq \Ex{X,\theta}{\Ex{Y}{
	\prod_{j \in J} (1- c_0 \nu \sin^2( \phi_j + x_j 2 \pi \theta ) }
}^{1+\nu} .
\end{align*}
By Lemma~\ref{lemma:techU2}, we can continue
		\begin{align*}
& = \Ex{Y,\theta}{\Ex{X}{
	\prod_{j \in J} (1- c_0 \nu \sin^2( \phi_j + x_j 2 \pi \theta ) }}^{1+\nu} \\
& \leq \Ex{Y,\theta}{\Ex{X}{
	\prod_{j \in J} (1- c_0 \nu \sin^2( \phi_j + x_j 2 \pi \theta ) }^{1+\nu}} \\
& \leq \Ex{X,Y,\theta}{\exp (- c \nu D(\theta))} ,
\end{align*}
where
$$D(\theta) = D_{x,y}(\theta) = \sum_{j \in G(x,y)}  \sin^2( 2 \pi \theta d_j).$$
%
%
By Lemma~\ref{lemma:encoding1}, $|J(y)| > n(\beta - 3 \lambda)$ for all but $2^{n(\beta - 2 \lambda)}$ choices for $y$. Similarly, $|J'(x)|> n(\alpha - 3 \lambda)$ for all but $2^{n(\alpha - 2 \lambda)}$  choices of $x$. 
By assumption,
$\beta - 3 \lambda + \alpha - 3 \lambda \geq \lambda$.
Since $|G(x,y)| \geq |J(y)| + |J'(x)| - n$, 
\begin{align*}
&\Pr [  |G(X,Y)| \leq \lambda n ]  \\
& \leq \Pr[|J(Y)| \leq n(\beta - 3 \lambda)] + \Pr[|J'(X)| \leq n(\alpha - 3 \lambda)]\\
& \leq 2^{-2\lambda n}+ 2^{-2\lambda n}.
\end{align*}

Next, we need a claim that is essentially identical to one used in the standard proof of Hal{\'a}sz inequality \cite{Taovu}:
\begin{claim*} 
Let $x,y$ be so that $G(x,y) \geq \lambda n$.
For every $0 \leq \rho\leq \tfrac{\lambda n}{4}$
and integer $\ell > 0$,
	$$\Pr_\theta  [D(\theta) < \rho ] \leq \frac{4 r_\ell(A)}{( \lambda n)^{2\ell+1/2}}   \sqrt{\rho} .$$
\end{claim*}
Given the claim, 
for every $x,y$ so that $G(x,y) \geq \lambda n$
and $\ell > 0$,
\begin{align*}
& \Ex{\theta}{ \exp (- c \nu D(\theta) ) } \\
& =
\int_0^1 \Pr_{\theta} [
 \exp (- c \nu D(\theta) ) > t  ] \, \mathrm{d} t
\\
& \leq \exp(- \tfrac{c \nu \lambda n}{4}) + 
\int_{\exp(- c \nu \lambda n /4)}^1 \Pr_{\theta} [
   D(\theta) < - \tfrac{\ln t}{c\nu} ] \, \mathrm{d} t
\\
& \leq \exp(-\tfrac{c \nu \lambda n}{4}) + 
\frac{4 r_\ell(A)}{(\lambda n)^{2\ell+1/2}}  \int_{0}^1 
 \sqrt{- \tfrac{ \ln t}{ c \nu}}
 \, \mathrm{d} t .
\end{align*}
The integral 
$\int_{0}^1 
 \sqrt{- \ln t}
 \, \mathrm{d} t \leq 1$ converges to a constant. 
For an appropriate $C = C(\beta,\delta) >0$
and $\ell >0$,
we get the desired contradiction.
 \begin{align*}
 \mu^{(1+\nu)^2}
 & \leq 2 \cdot 2^{-2 \lambda n} + 
 \exp(-\tfrac{c \nu \lambda n}{4}) +
\frac{4 r_\ell(A)}{\sqrt{c \nu} (\lambda n)^{2\ell+1/2}} \\
 & < 3  \exp(-\tfrac{ \nu n}{C}) +
 \tfrac{R_C(A)}{50 \sqrt{\nu}} .
 \end{align*}

\begin{proof}[Proof of Claim]
Let $G = G(x,y)$.
Observe that
\begin{align*}
& \Ex{\theta}{  ( |G|-2D(\theta) )^{2\ell} } \\
&= \Ex{\theta}{  \Big( \sum_{j \in G} \cos(4\pi d_j \theta)\Big)^{2\ell} }\\
&= 2^{-2\ell} \Ex{\theta}{   \Big( \sum_{j \in G} \exp(4\pi i d_j \theta)+ \exp(- 4\pi i  d_j \theta)\Big)^{2\ell} } \\
& \leq 2^{-2\ell} r_\ell(A) ;
\end{align*}
the last equality follows from the fact that of the $\leq (2|G|)^{2\ell}$ terms in the expansion, the only ones that survive are the ones with phase $0$. There are at most $r_\ell(A)$ such terms, and each contributes $1$. 

By Markov's inequality, since $|G| \geq \lambda n$,
\begin{align*}
\Pr_\theta  [D(\theta) \leq \tfrac{\lambda n}{4}  ] 
& \leq \Pr_\theta \Big [ (|G|-2D(\theta) )^{2\ell} \geq (\tfrac{\lambda n}{2})^{2\ell} \Big] \leq \frac{2^{-2\ell} r_\ell(A)}{(\lambda n/2)^{2\ell}} 
= \frac{r_\ell(A)}{(\lambda n)^{2\ell}}.
\end{align*}
This proves the claim for $\rho = \tfrac{\lambda n}{4}$.

It remains to prove the claim
for $\rho < \tfrac{\lambda n}{4}$.
This part uses Kemperman's theorem~\cite{kemperman}
from group theory (in fact Kneser's theorem~\cite{kneser}
for abelian groups suffices). Kemperman's theorem says that if a group is endowed with Haar measure $\mu$, then for any compact subsets $A,B$ of the group, $\mu(AB) \geq \min\{\mu(A) + \mu(B),1\}$.

Think of $[0,1)$ as the group $\R/\Z$.
Let 
$$S_\rho =  \{\theta \in \R/\Z : D(\theta) \leq \rho \}.$$

We claim that the $m$-fold sum $S_\rho+S_\rho+\dotsb+S_\rho
\subseteq \R/\Z$
is contained in $S_{\rho m^2}$. Indeed, 
\begin{align*}
|\sin(\eta_1+ \eta_2)| & = |\sin(\eta_1) \cos(\eta_2) + \sin(\eta_2) \cos(\eta_1)|\\
&\leq |\sin(\eta_1)|+ |\sin(\eta_2)|,
\end{align*}
and so 
\begin{align*}
\sin^2(\eta_1+\dotsb+ \eta_m) & \leq  (|\sin(\eta_1)|+ \dotsb + |\sin(\eta_m)|)^2\\ &\leq m  (\sin^2(\eta_1)+ \dotsb + \sin^2(\eta_m)) .
\end{align*}
It follows that
\begin{align*}
 D(\theta_1+\theta_2 + \cdots + \theta_m)
& \leq m (D(\theta_1) + D(\theta_2)+\cdots + D(\theta_m)) \\
& \leq m^2 \max \{D(\theta_1) , D(\theta_2),
\ldots,D(\theta_m)\} .
\end{align*}

Kemperman's theorem thus implies that 
$$|S_{\rho m^2}| \geq
|S_\rho+\dotsb+S_\rho| \geq m |S_\rho|,$$
as long as $S_{\rho m^2}$ is not
all of $\R/\Z$.
Since
\begin{align*}
\Ex{\theta}{D(\theta)}
= \sum_{j \in G} \Ex{\theta}{\sin^2(2 \pi \theta d_j)}
= \frac{|G|}{2} ,
\end{align*}
we can deduce that $|S_{\lambda n/4}| = \Pr_{\theta}[D(\theta) \leq \tfrac{\lambda n}{4}]$
is strictly less than one.
Hence, $S_{\lambda n/4}$ is not the full group $\R/\Z$.
Setting $m$ to be the largest integer so that $m^2\rho \leq \tfrac{\lambda n}{4}$, we can conclude
\begin{align*}
\Pr_\theta   [D(\theta) \leq \rho  ] &\leq \tfrac{1}{m} \Pr_\theta   [D(\theta) \leq \rho m^2  ] 
\leq \tfrac{1}{m} \Pr_\theta   [D(\theta) \leq \tfrac{\lambda n}{4}  ] .
\qedhere \end{align*}
\end{proof}
\end{proof}

\section{The lower bound for $\egh$} \label{sec:egh}

First, we show how to use Theorem \ref{thm:LBgen} to prove Theorem \ref{thm:LBgen1}.

\begin{proof}[Proof of Theorem~\ref{thm:LBgen1}]
The main observation is that for every integer $t$,
from a protocol that solves $\egh_{tn,tk}$ over the distribution $U_{tn,tk}$,
we get a randomized protocol that solves $\egh_{n,k}$. 
The reduction is constructed as follows.
Given inputs $x,y \in \{\pm 1\}^n$, first they repeat each input bit $t$ times to obtain $x',y' \in \{\pm 1\}^{tn}$. Then they sample a uniformly random $z \in \{\pm 1\}^{tn}$ using shared randomness, and compute $x'', y'' \in \{\pm 1\}^{tn}$ by setting $x''_j = x'_j  z_j$ and $y''_j = y'_j z_j$ for all $j \in [n]$. 
	Finally, they randomly permute the coordinates of $x'',y''$ to obtain $x''',y'''$. The result is that $x''',y'''$ are uniformly distributed among all inputs with inner product 
	that is equal to $t$ times the inner product of $x,y$.
The pair $(x''',y''')$ was generated with no communication.
Finally, they run the protocol for $\egh_{tn,tk}$ on $x''',y'''$.

	Now, let $\alpha,n_0$ be the constants from Theorem \ref{thm:LBgen}. 
	Let $t >0$ and $n > n_0$ be integers so that
	both $n/t$ and $k=\sqrt{n}/t$ are even
	and $k \leq \alpha \sqrt{n/t}$.
	By Theorem \ref{thm:LBgen},
	any protocol for $\egh_{n/t,k}$ over $U_{n/t,k}$ requires $\Omega(n/t)$ communication.
	By the reduction above,
	any protocol for $\egh_{n} = \egh_{n,\sqrt{n}}$
	yields a protocol for $\egh_{n/t,k}$.
%
%
\end{proof}

\begin{proof}[Proof of Theorem~\ref{thm:LBgen}]
	Suppose the assertion of the theorem is false.
	By a standard argument in communication complexity, the space of inputs can be partitioned into rectangles $R_1, \dotsc, R_L$ with $L \leq 2^{(1-\beta) n}$, where the output of the protocol on each $R_\ell$ is fixed. 
	
	Let $X,Y$ be i.i.d.\ uniformly at random in $\{\pm 1\}^n$. 
	Let $E$ denote the event that $|\inner{X}{Y} | = k$. 
	Define the collection of ``typical'' rectangles as
	$$\good = \Big \{ \ell \in [L] \ : \ \Pr_{X,Y}[E | R_\ell] \geq \tfrac{\Pr_{X,Y}[E]}{10}
	\quad \& \quad \Pr_{X,Y}[R_\ell] \geq 2^{-\big(1-\tfrac{\beta}{2}\big)n}\Big\}.$$
	For $\alpha \leq 2$, 
	because $k=n \mod 2$, we have $\Pr_{X,Y}[E] \geq \tfrac{p}{\sqrt{n}}$
	for some universal constant $p>0$. 
	The contribution of non-typical rectangles is small:
	\begin{align*}
	\sum_{\ell \not \in \good}\Pr_{X,Y}[R_\ell|E] & = \tfrac{1}{\Pr_{X,Y}[E]} \sum_{\ell \not \in \good}\Pr_{X,Y}[R_\ell] \Pr_{X,Y}[E|R_\ell] \\
	& < \tfrac{1}{\Pr_{X,Y}[E]} \Big(L 2^{-\big(1-\tfrac{\beta}{2}\big)n} +\tfrac{\Pr_{X,Y}[E]}{10} \Big) < \tfrac{1}{5} ,
	\end{align*}	
	for $n$ large enough.
	Because $k = -k \mod 4$
	and $|k| < \alpha \sqrt{n}$,
	for each $\ell \in \good$, Theorem~\ref{thm:mainP} with $\eps
	\geq \tfrac{\beta}{2}$ implies that
	\begin{align*}
	&|\Pr_{X,Y}[\inner{X}{Y} = k | R_\ell \wedge E] - \Pr_{X,Y}[\inner{X}{Y} = - k | R_\ell \wedge E] |\\
	&= |\Pr_{X,Y}[\inner{X}{Y} = k | R_j ] - \Pr_{X,Y}[\inner{X}{Y} = - k | R_j ]| \cdot \tfrac{1}{\Pr_{X,Y}[E|R_j]} \\
	& \leq  \alpha \sqrt{n}  \tfrac{c_0 }{n} \cdot \tfrac{10\sqrt{n}}{p} < \tfrac{1}{6} ,
	\end{align*}
	for $\alpha$ small enough.
	So, the probability of error conditioned on $R_\ell$ for $\ell \in \good$ is at least $\tfrac{5}{12}$.
	The total probability of error is at least
	\begin{align*}
	\sum_{\ell \in \good} \Pr_{X,Y}[R_\ell|E] \cdot \tfrac{5}{12} > 
	\tfrac{4}{5} \cdot \tfrac{5}{12}  = \tfrac{1}{3}.
	\end{align*}
	This contradicts the correctness of the protocol. \qedhere
	
\end{proof}

\subsubsection*{Acknowledgements}
We wish to thank James Lee, Oded Regev, Avishay Tal
and David Woodruff for helpful conversations.
We also wish to thank the two anonymous reviewers for
many valuable comments.

%
\bibliography
\appendix
\section{Strict Convexity} \label{sec:techclaim}
\begin{proof}[Proof of Claim \ref{clm:MinPhi}]
If $\alpha_1 =0$, then the claim is trivially true. So, 
assume that $\alpha_1 >0$.  Without loss of generality, we may also assume that $\kappa>0$ is small enough so that $4^\kappa > \exp(\kappa + \kappa^2)$. 

Let $p = \Pr[W = 1] \in [\kappa, 1-\kappa]$ and $\xi = \tfrac{\alpha_{-1}}{\alpha_1} \in [0,\tfrac{1}{2}]$. 
So, 
\begin{align*}
\frac{\Expect{\alpha_W}^{1+\nu}}{\Expect{\alpha_W^{1+\nu}}} 
& = \frac{(p+ (1-p) \xi)^{1+\nu}}{p  + (1-p) \xi^{1+\nu}}. 
\end{align*}
We need to upper bound this ratio by $1-c_1 \nu$, for some constant $c_1$ that depends only on $\kappa$. 
Let
	$$\Phi(\xi,p,\nu)
	= (p+(1-p)\xi^{1+\nu}) - (p+(1-p)\xi)^{1+\nu}.$$
	We shall argue that there is a constant $c_1 = c_1(\kappa)>0$ such that $\Phi(\xi,p,\nu) \geq  c_1\nu$. This completes the proof, since 
	\begin{align*}
	&\frac{(p + (1-p) \xi)^{1+\nu}}{(p + (1-p) \xi^{1+\nu})} = 1 - \frac{\Phi(\xi,p,\nu)}{(p + (1-p) \xi^{1+\nu})} < 1-c_1\nu.
	\end{align*}

First,	we show that for every $\nu$ and $\xi$, 
	the function $\Phi(\xi,p,\nu)$ is minimized when $p = \kappa$. Consider
	\begin{align*}
	\frac{\partial \Phi}{\partial p} & = 1 - \xi^{1+\nu} - (1+\nu) (p+(1-p)\xi)^{\nu} (1-\xi) \\
	&\geq 1- \xi^{1+\nu}- (1+\nu) (1-\xi)\\
	&\geq \xi (1+\nu - \xi^{\nu})>0,
	\end{align*}
	since $\xi^{\nu}<1$.
	So, the minimum is achieved when $p = \kappa$.

Second, we claim that for every $\nu$ and $p$, 
	the function $\Phi(\xi,p,\nu)$ is minimized when $\xi = \tfrac{1}{2}$. Consider
	\begin{align*}
	\frac{\partial \Phi}{\partial \xi} & = (1-p)(1+\nu)\xi^{\nu} - (1+\nu)(p+(1-p)\xi)^{\nu} (1-p)\\
	&= (1-p)(1+\nu) (\xi^\nu - (p+ (1-p)\xi)^\nu) <0,
	\end{align*}
	since $p+ (1-p)\xi > \xi$. So, the minimum is achieved when $\xi = 1/2$. 

Third, we control the derivative with respect to $\nu$
for $\xi = \tfrac{1}{2}$ and $p = \kappa$.
Consider
	\begin{align*}
	\frac{\partial  \Phi}{\partial \nu}(\tfrac{1}{2},\kappa,\nu) &=  (1-\kappa) \ln(\tfrac{1}{2}) (\tfrac{1}{2})^{1+\nu} - \ln (\tfrac{1+\kappa}{2}) (\tfrac{1+\kappa}{2})^{1+\nu} \\
	& \geq (\tfrac{1}{2})^2 ((1-\kappa) \ln(\tfrac{1}{2}) - \ln (\tfrac{1+\kappa}{2}) (1+\kappa)^{1+\nu}),  
	\end{align*}since $\nu \leq 1$. The expression 
	$$(1-\kappa) \ln(\tfrac{1}{2}) - \ln (\tfrac{1+\kappa}{2}) (1+\kappa)^{1+\nu}$$ 
	only increases with $\nu$. When $\nu=0$, this expression is
	\begin{align*}
\ln(\tfrac{2^{2\kappa}}{(1+\kappa)^{1+\kappa}}) \geq  \ln(\tfrac{4^{\kappa}}{\exp(\kappa(1+\kappa))})>0,
	\end{align*}
	since $4^\kappa > \exp(\kappa + \kappa^2)$. This proves that $\frac{\partial  \Phi}{\partial \nu}(\tfrac{1}{2},\kappa,\nu)>c_1$ for some constant $c_1(\kappa) >0$. 
  
Finally,
   $$\Phi(\xi,p,\nu) \geq \Phi(\tfrac{1}{2},\kappa,\nu)
   = \int_0^\nu 	\frac{\partial  \Phi}{\partial \nu}(\tfrac{1}{2},\kappa,\zeta) \, \mathrm{d} \zeta
   \geq \int_0^\nu c_1 \, \mathrm{d} \zeta 
   = c_1 \nu .\qedhere $$
\end{proof}

\end{document}